\documentclass[11pt,a4paper,reqno,intlimits]{amsart}

\usepackage{amssymb}
\usepackage{color}
\usepackage[mathscr]{euscript}
\usepackage{xspace}
\usepackage{epsfig,rotating}
\usepackage{enumerate}
\usepackage[round]{natbib}

\usepackage[bookmarksopen,pdfstartview=FitH]{hyperref}
\hypersetup{colorlinks,%
            citecolor=blue,%
            filecolor=blue,%
            linkcolor=blue,%
            urlcolor=blue}

\theoremstyle{plain}
\newtheorem{theorem}{Theorem}[section]
\newtheorem{lemma}[theorem]{Lemma}
\newtheorem{corollary}[theorem]{Corollary}

\theoremstyle{definition}

\newtheorem{example}[theorem]{Example}

\newtheorem{remark}[theorem]{Remark}

\newtheorem*{assud}{Assumption ($\mathbb{D}$)}
\newcommand{\scal}[2]{\ensuremath{\langle #1, #2 \rangle}}

\def\E{\mathrm{I\kern-.2em E}}

\def\P{\mathrm{I\kern-.2em P}}
\def\R{\mathbb{R}}
\def\C{\mathbb{C}}
\def\F{\mathcal{F}}

\def\Rn{\mathbb{R}^n}

\def\ud{\mathrm{d}}
\def\dy{\mathrm{d}y}

\def\dv{\mathrm{d}v}

\def\e{\mathrm{e}}
\def\ii{\mathrm{i}}

\def\lev{L\'evy\xspace}

\newcommand{\prazess}[1][X]{{\ensuremath{#1=(#1_t)_{t\ge 0}}}\xspace}

\numberwithin{equation}{section}

\begin{document}

\title{Computation of copulas by Fourier methods}
\author{Antonis Papapantoleon}
\address{Institute of Mathematics, TU Berlin, Stra\ss e des 17. Juni 136,
         10623 Berlin, Germany}
\email{papapan@math.tu-berlin.de}

\thanks{The author thanks E. A. von Hammerstein for helpful comments and
        suggestions.}
\date{}\maketitle\pagestyle{myheadings}\frenchspacing\setcounter{tocdepth}{1}

\begin{abstract}
We provide an integral representation for the (implied) copulas of dependent 
random variables in terms of their moment generating functions. The proof uses 
ideas from Fourier methods for option pricing. This representation can be used 
for a large class of models from mathematical finance, including L\'evy and 
affine processes. As an application, we compute the implied copula of the NIG 
L\'evy process which exhibits notable time-dependence.
\end{abstract}

\section{Introduction}

Copulas provide a complete characterization of the dependence structure between 
random variables and link in a very elegant way the joint distribution with the
marginal distributions via Sklar's Theorem. However, they are a rather static 
concept and do not blend well with stochastic processes which can be used to
describe the random evolution of dependent quantities, e.g. the evolution of
several stock prices. Therefore other methods to create dependence in stochastic
models have been developed. Multivariate stochastic processes spring 
immediately to mind, for example L\'evy or affine processes (cf. e.g.
\citealt{Sato99}, \citealt*{DuffieFilipovicSchachermayer03},
\citealt*{Cuchiero_Filipovic_Mayerhofer_Teichmann_2011} or \citealt*{MuhleKarbe_Pfaffel_Stelzer_2012}), while in mathematical 
finance models using time-changes or linear mixture models have been developed; 
see e.g. \citet{LucianoSchoutens06}, \citet{Luciano_Semeraro_2010}, 
\citet{Kawai09}, \citet{EberleinMadan09} or \citet{Khanna_Madan:2009}, to 
mention just a small part of the existing literature. In these approaches 
however the copula is typically not known explicitly. Another very interesting 
approach is due to \citet{KallsenTankov04}, who introduced L\'evy copulas to 
characterize the dependence structure of L\'evy processes.

In this note, we provide a new representation for the (implied) copula of a 
multidimensional random variable in terms of its moment generating function. The 
derivation of the main result borrows ideas from Fourier methods for option 
pricing, and the motivation stems from the knowledge of the moment generating 
function in most of the aforementioned models. This paper is organized as 
follows: in Section \ref{copula} we provide the representation of the copula in 
terms of the moment generating function; the results are proved for random 
variables for simplicity, while stochastic processes are considered as a 
corollary. In Section \ref{examples} we provide two examples to showcase how 
this method can be applied, for example, in performing sensitivity analysis of 
the copula with respect to the parameters of the model. Finally, Section 
\ref{remarks} concludes with some remarks.

\section{Copulas via Fourier transform methods}
\label{copula}

Let $\Rn$ denote the $n$-dimensional Euclidean space, $\scal{\cdot}{\cdot}$ the 
Euclidean scalar product and $\R^n_{-}$ the negative orthant, i.e. 
$\R^n_{-}=\{x\in\R^n: x_i<0 \,\, \forall i\}$. We consider a random variable 
$X=(X_1,\dots,X_n)^\top\in\Rn$ defined on a probability space $(\Omega,\F,\P)$. 
We denote by $F$ the cumulative distribution function (cdf) of $X$ and by $f$ 
its probability density function (pdf). Let $C$ denote the copula of $X$ and $c$ 
its copula density function. Analogously, let $F_i$ and $f_i$ denote the cdf and 
pdf respectively of the marginal $X_i$, for all $i\in\{1,\dots,n\}$. In 
addition, we denote by $F_i^{-1}$ the generalized inverse of $F_i$, i.e. 
$F_i^{-1}(u)=\inf\{v\in\R: F_i(v)\ge u\}$.

We denote by $M_X$ the (extended) moment generating function of $X$:
\begin{align}\label{MGF}
M_X(u) = \E\big[\e^{\scal{u}{X}}\big],
\end{align}
for all $u\in\C^n$ such that $M_X(u)$ exists. Let us also define the set 
\begin{align*}
\mathcal{I} = 
\big\{ R\in\Rn: M_X(R)<\infty \text{ and } M_X(R+\ii\cdot)\in L^1(\Rn) \big\}.
\end{align*}
In the sequel, we will assume that the following condition is in force.

\begin{assud}
$\mathcal{R}:=\mathcal{I}\cap\Rn_{-}\neq\emptyset$.
\end{assud}

\begin{remark}
The integrability of the moment generating function required by Assumption 
$(\mathbb{D})$ has the following implications:
\begin{enumerate}[(a)]
\item the distribution function $F$ is absolutely continuous with respect to 
      the Lebesgue measure;
\item the density function $f$ is bounded and continuous;
\item the marginal distribution functions $F_i$ are also absolutely continuous.      
\end{enumerate}
See \citet[Proposition 2.5]{Sato99} for (a) and (b) and \citet[Theorem 
12.2]{Jacod_Protter_2003} for (c).
\end{remark}

\begin{theorem}\label{copula-thm}
Let $X$ be a random variable that satisfies Assumption $(\mathbb{D})$. The 
copula of $X$ is provided by
\begin{align}\label{eq-cop}
C(u)
 &= \frac{1}{(-2\pi)^n} \int_{\R^n} M_{X}(R+\ii v)
     \frac{\e^{-\langle R+\ii v,x\rangle}}
    {\prod_{i=1}^n(R_i+\ii v_i)}\dv \Big|_{x_i=F_i^{-1}(u_i)},
\end{align}
where $u\in[0,1]^n$ and $R\in\mathcal{R}$.
\end{theorem}

\begin{proof}
Assumption $(\mathbb{D})$ implies that $F_1,\dots,F_n$ are continuous and we 
know from Sklar's Theorem that the copula of $X$ is unique and provided by 
\begin{align}\label{Sklar}
C(u_1,\dots,u_n)
 = F\big(F_1^{-1}(u_1),\dots,F_n^{-1}(u_n)\big);
\end{align}
see e.g. \citet*[Theorem 5.3]{McNeilFreyEmbrechts05} for a proof in this 
setting and \citet{Rueschendorf_2009} for an elegant proof in the general case.

We will evaluate the joint cdf $F$ using the methodology of Fourier methods for
option pricing. That is, we will think of the cdf as the `price' of a digital
option on several fictitious assets. Let us define the function
\begin{align}
g(y) = 1_{\{y_1\le x_1,\dots,y_n\le x_n\}}(y),
\quad x,y\in\Rn,
\end{align}
and denote by $\widehat{g}$ its Fourier transform. Then we have that
\begin{align}\label{cdf-Fourier}
F(x)
 &= \P(X_1\le x_1,\dots,X_n\le x_n) \nonumber\\
 &= \E\big[1_{\{X_1\le x_1,\dots,X_n\le x_n\}}\big]
  = \E[g(X)] \nonumber\\
 &= \frac{1}{(2\pi)^n} \int_{\R^n} M_{X}(R+\ii v) \widehat{g}(\ii R-v)\dv,
\end{align}
where we have applied Theorem 3.2 in \cite*{EberleinGlauPapapantoleon08}. The
prerequisites of this theorem are satisfied due to Assumption $(\mathbb{D})$ 
and because $g_R\in L^1(\Rn)$, where $g_R(x):=\e^{-\langle R,x\rangle}g(x)$ for
$R\in\Rn_{-}$.

Finally, the statement follows from \eqref{Sklar} and \eqref{cdf-Fourier} once
we have computed the Fourier transform of $g$. We have for $R_i<0$, 
$i\in\{1,\dots,n\}$,
\begin{align}
\widehat{g}(\ii R-v)
 &= \int_{\R^n} \e^{\ii\langle \ii R-v,y\rangle} g(y)\dy \nonumber\\
 &= \int_{\R^n} \e^{\ii\langle \ii R-v,y\rangle}
     1_{\{y_1\le x_1,\dots,y_n\le x_n\}}\dy \nonumber\\
 &= \prod_{i=1}^n \int_{-\infty}^{x_i} \e^{(-R_i-\ii v_i)y_i}
    \ud y_i\nonumber\\
 &= (-1)^n \prod_{i=1}^n 
     \frac{\e^{-(R_i+\ii v_i)x_i}}{R_i+\ii v_i},
\end{align}
which concludes the proof.
\end{proof}

\begin{remark}
If the moment generating function of the marginals is known, the inverse
function can be easily computed numerically. We have that
\begin{align*}
F_i^{-1}(u)
 &=\inf\{v\in\R: F_i(v)\ge u\} \\
 &= \inf\{v\in\R: \E\big[1_{\{X_i\le v\}}\big]\ge u\},
\end{align*}
where the expectation can be computed using \eqref{cdf-Fourier} again, while a
root finding algorithm provides the infimum (using the continuity of $F_i$). 
\end{remark}

We can also compute the copula density function using Fourier methods, which
resembles the computation of Greeks in option pricing.

\begin{lemma}\label{copula-lemma}
Let $X$ be a random variable that satisfies Assumption $(\mathbb{D})$ and 
assume further that the marginal distribution functions $F_1,\dots,F_n$ are 
strictly increasing and continuously differentiable. Then, the copula density 
function $c$ of $X$ is provided by
\begin{align}\label{eq-cop-dens}
c(u)
 &= \frac{1}{(2\pi)^n \prod_{i=1}^nf_i(x_i)}
     \int_{\R^n} M_X(R+\ii v)\, \e^{-\langle R+\ii v,x\rangle}\dv
      \Big|_{x_i=F_i^{-1}(u_i)},
\end{align}
where $u\in(0,1)^n$ and $R\in\mathcal{R}$.
\end{lemma}

\begin{proof}
The distribution functions $F$ and $F_1,\dots,F_n$ are absolutely continuous 
hence the copula density exists, cf. \citet[p.~197]{McNeilFreyEmbrechts05}. Let 
$u\in(0,1)^n$, then we have that $x_i=F_i^{-1}(u_i)$ is finite for every 
$i\in\{1,\dots,n\}$, hence $\e^{-\langle R,x\rangle}$ is bounded. Using 
Assumption $(\mathbb{D})$ we get that the function $M_X(R+\ii v) \e^{-\langle 
R+\ii v,x\rangle}$ is integrable and we can interchange differentiation and 
integration. Then we have that
\begin{align}\label{cop-dens-step-1}
c(u)
 &= \frac{\partial^n}{\partial u_1\dots\partial u_n} C(u_1,\dots,u_n)
    \nonumber\\
 &= \frac{\partial^n}{\partial u_1\dots\partial u_n}
    \frac{1}{(-2\pi)^n} \int_{\R^n} M_{X}(R+\ii v)
     \frac{\e^{-\langle R+\ii v,x\rangle}}
     {\prod_{i=1}^n(R_i+\ii v_i)}\Big|_{x_i=F_i^{-1}(u_i)}\dv
    \nonumber\\
 &= \frac{1}{(-2\pi)^n} \int_{\R^n} 
     \frac{M_{X}(R+\ii v)}{\prod_{i=1}^n(R_i+\ii v_i)}
     \frac{\partial^n}{\partial u_1\dots\partial u_n}
      \e^{-\langle R+\ii v,x\rangle}\Big|_{x_i=F_i^{-1}(u_i)}\dv.
\end{align}
Now, since the marginal distribution functions are continuously differentiable,
using the chain rule and the inverse function theorem we get that
\begin{multline}\label{cop-dens-step-2}
\frac{\partial^n}{\partial u_1\dots\partial u_n}
\left(\e^{-\langle R+\ii v,x\rangle} \Big|_{x_i=F_i^{-1}(u_i)}\right)\\
 = (-1)^n \prod_{i=1}^n(R_i+\ii v_i) \e^{-\langle R+\ii v,x\rangle}
    \frac{1}{\prod_{i=1}^n f_i(x_i)}\Big|_{x_i=F_i^{-1}(u_i)},
\end{multline}
which combined with \eqref{cop-dens-step-1} yields the required result.
\end{proof}

A natural application of these representations is for the calculation of the 
copula of a random variable $X_t$ from a multidimensional stochastic process 
$X=(X_t)_{t\ge0}$. There are many examples of stochastic processes where the 
corresponding characteristic functions are known explicitly. Prominent examples 
are \lev processes, self-similar additive (`Sato') processes and affine 
processes.

\begin{corollary}
Let \prazess be an $\Rn$-valued stochastic process on a basis
$(\Omega,\F,(\F_t)_{t\ge0},\P)$. Assume that the random variable $X_t$, $t\ge0$,
satisfies Assumption $(\mathbb{D})$. Then, the copula of $X_t$ is provided by
\begin{align}
C_t(u)
 &= \frac{1}{(-2\pi)^n} \int_{\R^n} M_{X_t}(R+\ii v)
     \frac{\e^{-\langle R+\ii v,x\rangle}}
          {\prod_{i=1}^n(R_i+\ii v_i)}\dv
\Big|_{x_i=F_{X_t^i}^{-1}(u_i)},
\end{align}
where $u\in[0,1]^n$ and $R\in\mathcal{R}$. An analogous statement holds for the
copula density function $c_t$ of $X_t$.
\end{corollary}

\section{Examples}
\label{examples}

We will demonstrate the applicability and flexibility of Fourier methods for 
the computation of copulas using two examples. First we consider a 2D normal 
random variable and next a 2D normal inverse Gaussian (NIG) \lev process. 
Although the copula of the normal random variable is the well-known Gaussian 
copula, little was known about the copula of the NIG distribution until 
recently; see Theorem 5.13 in \citet{Schmidt03} for a special case. 
\citet[Chapter 2]{Hammerstein_2011} has now provided a general characterization
of the (implied) copula of the multidimensional NIG distribution using properties
of normal mean-variance mixtures.

\begin{example}\label{ex-Gaussian}
The first example is simply a `sanity check' for the proposed method. We
consider the 2-dimensional Gaussian distribution and compute the corresponding
copula for correlation values equal to $\rho=\{-1,0,1\}$; see Figure
\ref{normal} for the resulting contour plots. Of course, the copula of this
example is the Gaussian copula, which for correlation coefficients equal to
$\{-1,0,1\}$ corresponds to the \emph{countermonotonicity} copula, the
\emph{independence} copula and the \emph{comonotonicity} copula respectively.
This is also evident from Figure \ref{normal}.

\begin{figure}
 \centering
  \includegraphics[width=4.cm]{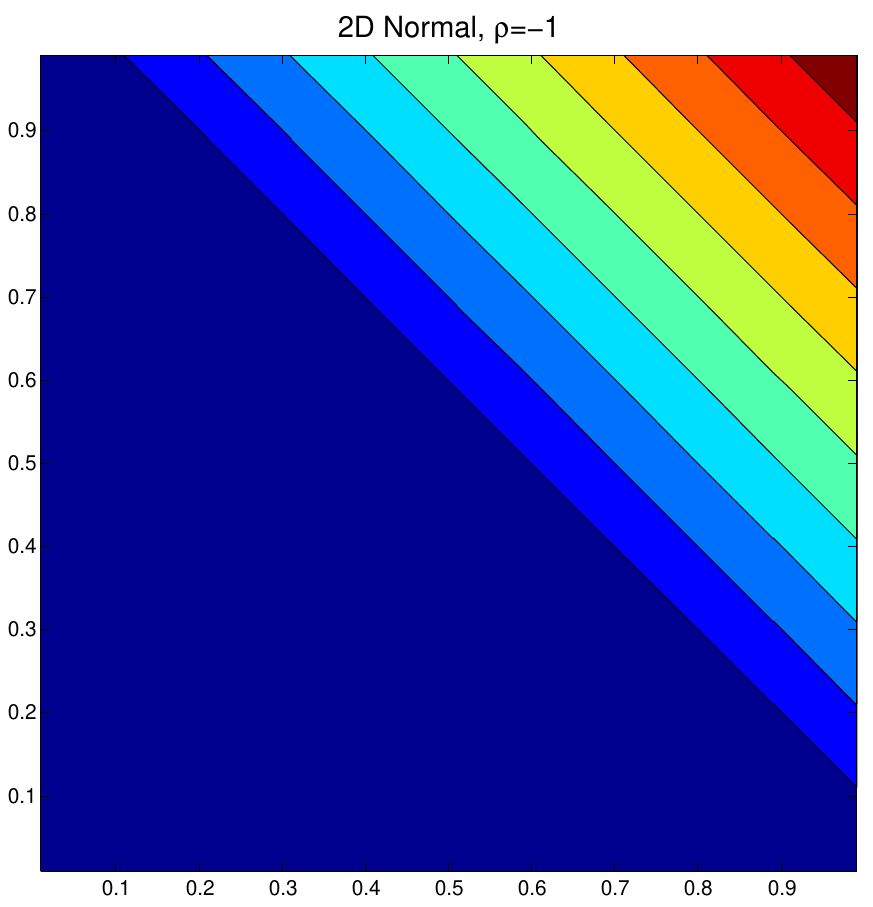}
  \includegraphics[width=4.cm]{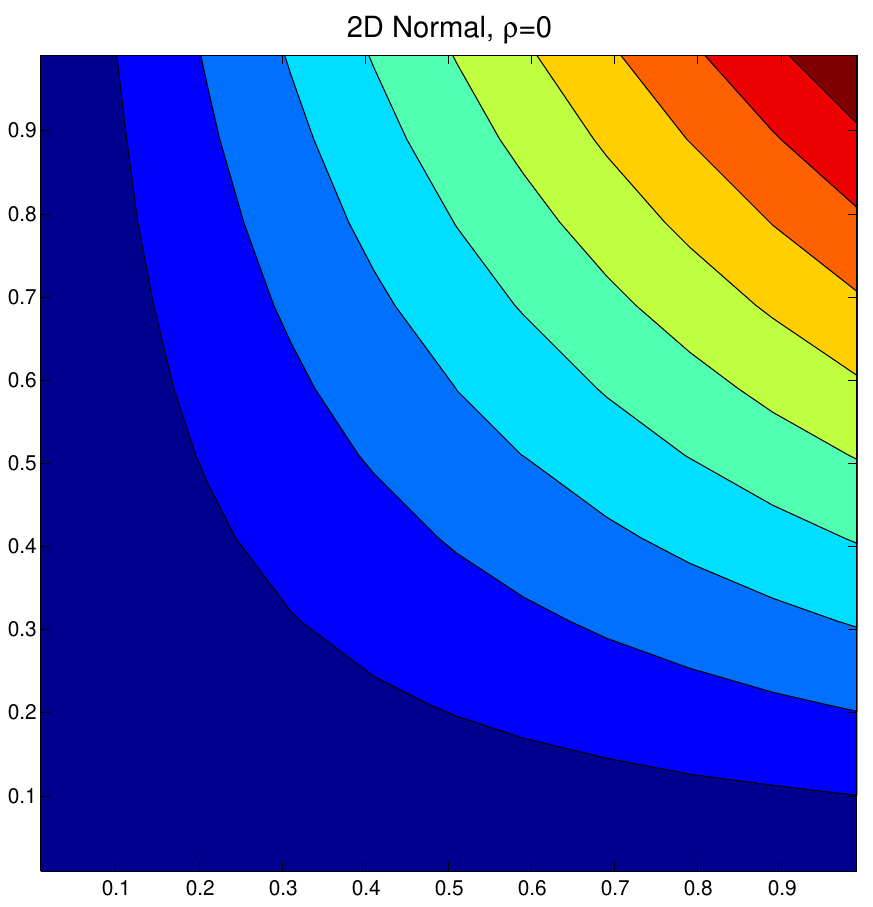}
  \includegraphics[width=4.cm]{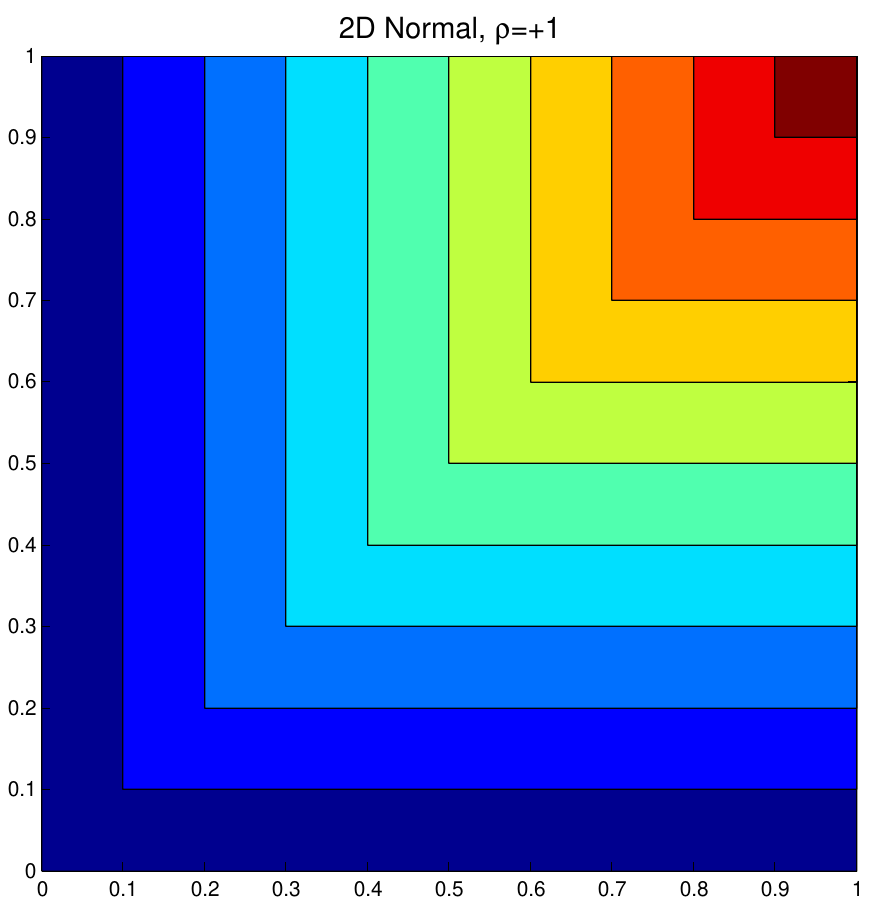}
 \caption{Contour plots of copulas for Example \ref{ex-Gaussian}.}
 \label{normal}
\end{figure}
\end{example}

\begin{example}
Let \prazess be a 2-dimensional NIG \lev process, i.e.
\begin{align}
X_t=(X_t^1,X_t^2)\sim\text{NIG}_2(\alpha,\beta,\delta t,\mu t,\Delta),
\quad t\ge0.
\end{align}
The parameters satisfy: $\alpha,\delta>0$, $\beta,\mu\in\R^2$, and
$\Delta\in\R^{2\times2}$ is a symmetric, positive definite matrix (w.l.o.g. we
can assume $\text{det}(\Delta)=1$). Moreover,
$\alpha^2>\scal{\beta}{\Delta\beta}$. The moment generating function of $X_1$,
for $u\in\R^2$ with $\alpha^2-\scal{\beta+u}{\Delta(\beta+u)}\ge0$, is
\begin{align*}
M_{X_1}(u)
 &= \exp\left( \langle u,\mu\rangle 
       + \delta\left(\sqrt{\alpha^2-\langle \beta,\Delta\beta\rangle}
          -\sqrt{\alpha^2-\langle\beta+u,\Delta(\beta+u)\rangle}\right)\right),
\end{align*}
cf. \citet{Barndorff-Nielsen98}. The marginals are also NIG distributed and we
have that
$X_t^i\sim\text{NIG}(\hat\alpha^i,\hat\beta^i,\hat\delta^it,\hat\mu^it)$, where
\begin{align*}
\hat\alpha^i =
\sqrt{\frac{\alpha^2-\beta_j^2(\delta_{jj}-\delta_{ij}^2\delta_{ii}^{-1})}
     {\delta_{ii}}},\,\,
\hat\beta^i  = \beta_i+\beta_j\delta_{ij}^2\delta_{ii}^{-1},\,\,
\hat\delta^i = \delta\sqrt{\delta_{ii}},\,\,
\hat\mu^i    = \mu_i, 
\end{align*}
for $i=\{1,2\}$ and $j=\{2,1\}$; cf. e.g. \citet[Theorem~1]{Blaesild81}.
Assumption $(\mathbb{D})$ is satisfied for $R\in\R^2_{-}$ such that
$\alpha^2-\scal{\beta+R}{\Delta(\beta+R)}\ge0$; see Appendix B in
\citet{EberleinGlauPapapantoleon08}. Hence $\mathcal{R}\ne\emptyset$.

Therefore, we can apply Theorem \ref{copula-thm} to compute the copula of the
NIG distribution. The parameters used in the numerical example are similar to
\citet[pp.~233-234]{EberleinGlauPapapantoleon08}: $\alpha=10.20$,
$\beta=\bigl(\begin{smallmatrix}-3.80\\-2.50\end{smallmatrix}\bigr)$,
$\delta=0.150$, $\mu\equiv0$, and two matrices
$\Delta^+=\bigl(\begin{smallmatrix}1&0\\0&1\end{smallmatrix}\bigr)$ and
$\Delta^-=\bigl(\begin{smallmatrix}1&-1\\-1&2\end{smallmatrix}\bigr)$, which
lead to positive and negative correlation. The correlation
coefficients are $\rho_+=0.1015$ and $\rho_-=-0.687$ respectively. 

The contour plots are exhibited in Figures \ref{NIG_T=1} and \ref{NIG_T=12} and 
show clearly the influence of the different mixing matrices $\Delta^+$ and 
$\Delta^-$ to the dependence structure. Moreover, we can also observe that time 
has a significant effect on the dependence structure of the multidimensional 
NIG L\'evy process. This is an interesting observation, since the correlation 
matrix is invariant over time (which is true for any L\'evy process).

\begin{figure}[h]
 \centering
  \includegraphics[width=5.75cm]{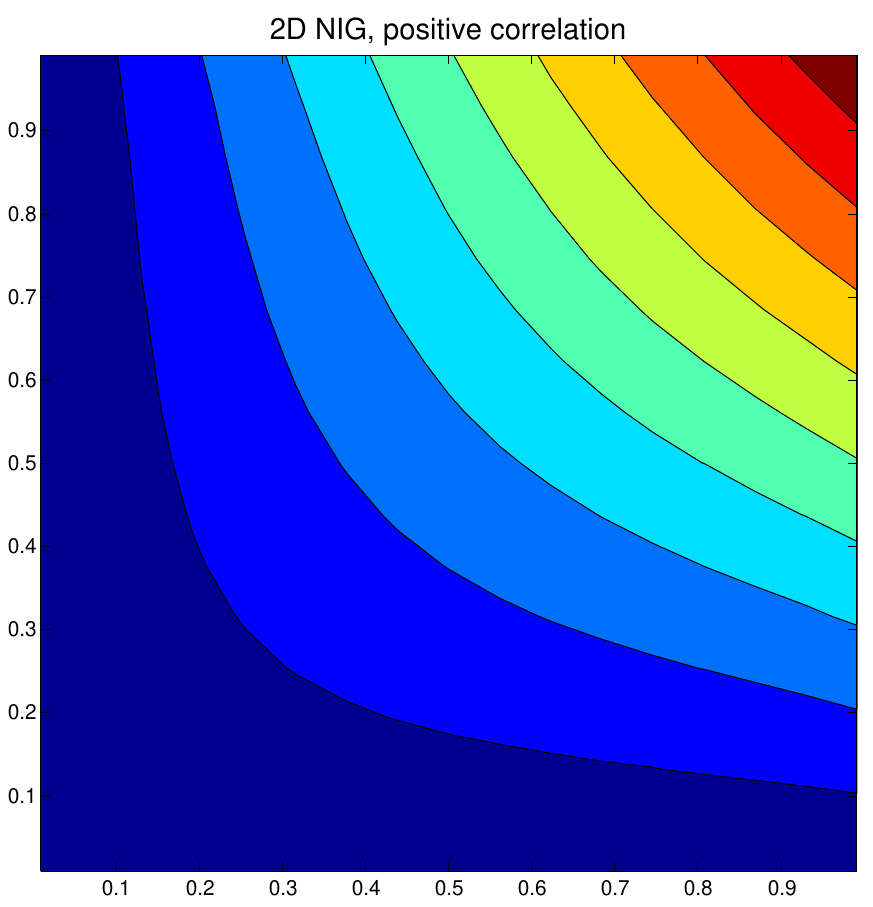}
  \includegraphics[width=5.75cm]{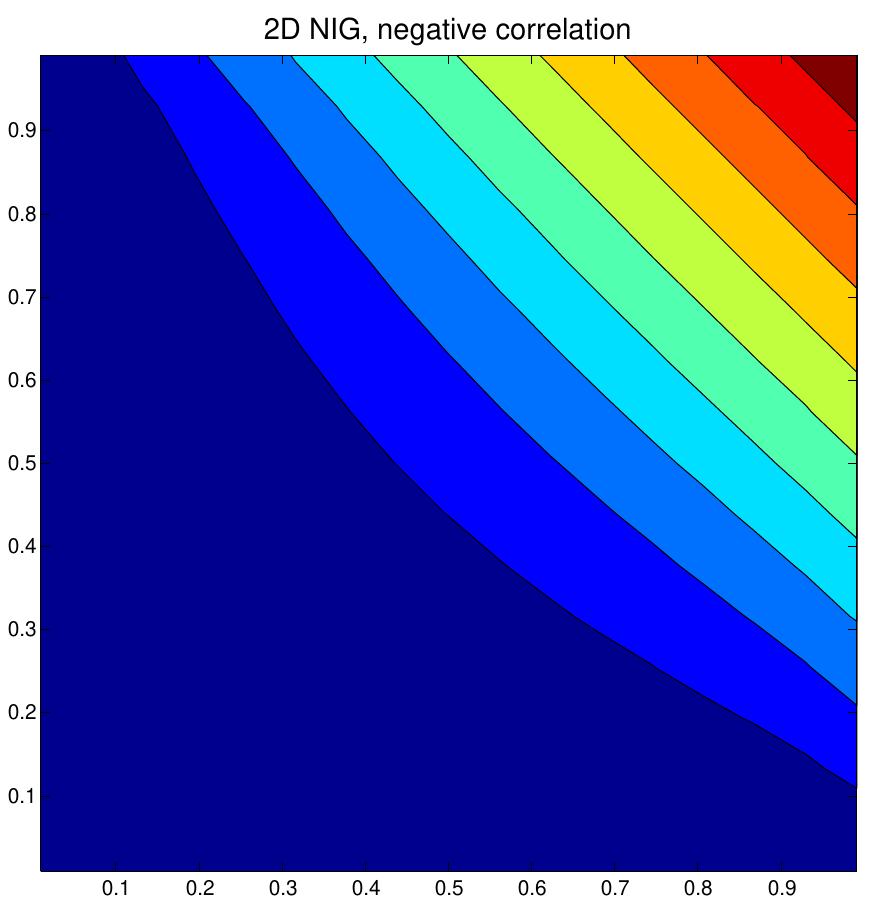}
 \caption{Contour plots of copulas for NIG, $t=1$.}
 \label{NIG_T=1}
\end{figure}
\begin{figure}[h]
 \centering
  \includegraphics[width=5.75cm]{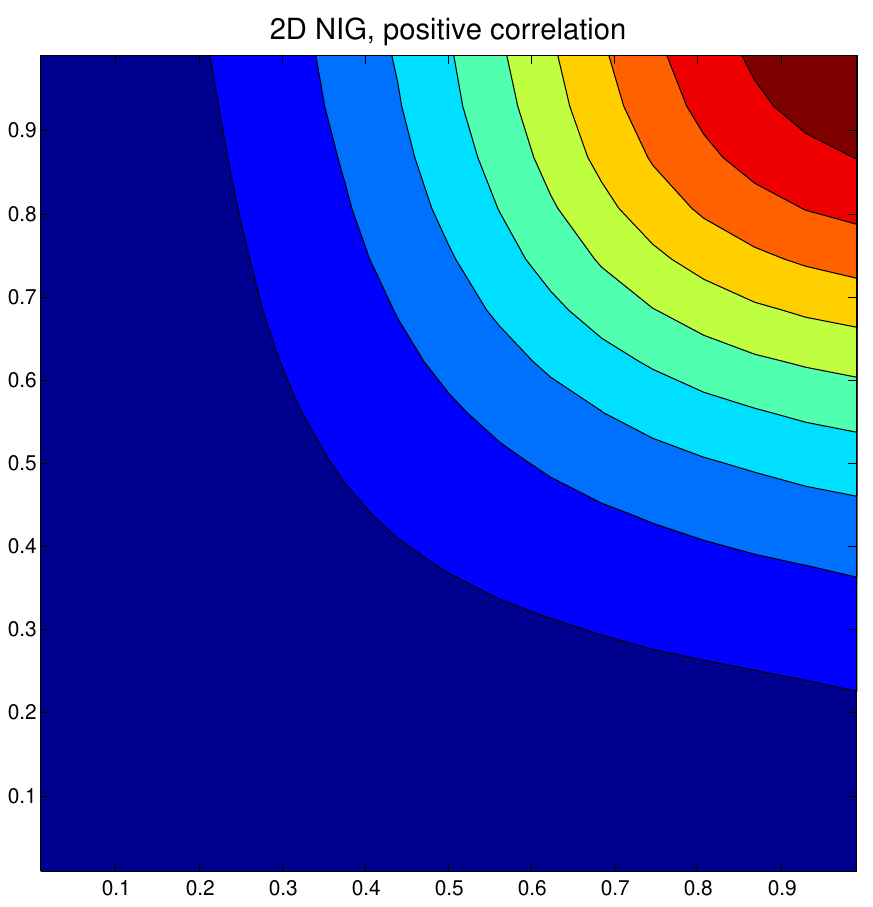}
  \includegraphics[width=5.75cm]{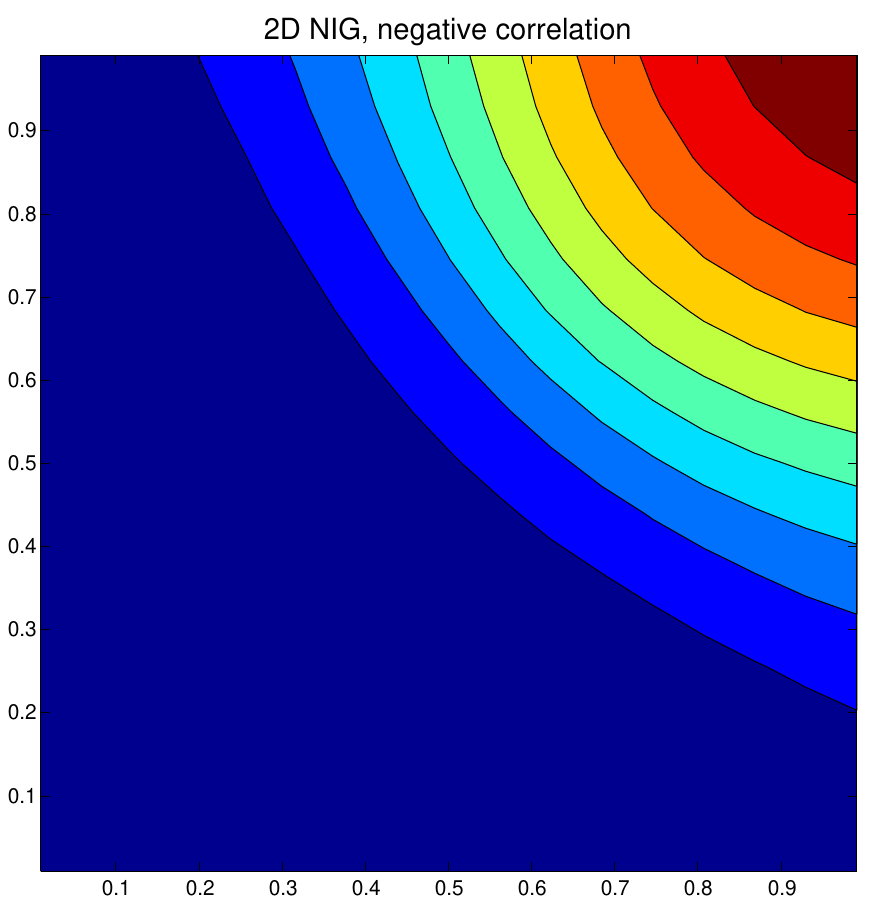}
 \caption{Contour plots of copulas for NIG, $t=\frac12$.}
 \label{NIG_T=12}
\end{figure}
\end{example}

\section{Final remarks}
\label{remarks}

We will not elaborate on the speed of Fourier methods compared with Monte Carlo
methods in the multidimensional case; the interested reader is refered to
\citet{HurdZhou09} for a careful analysis. Moreover, \citet{Villiger_2007}
provides recommendations on the efficient implementation of Fourier integrals
using sparse grids in order to deal with the `curse of dimensionality'. Let us 
point out though that the computation of the copula function will be much 
quicker than the computation of the copula density, since the integrand in
\eqref{eq-cop} decays much faster than the one in \eqref{eq-cop-dens}. One
should think of the analogy to option prices and option Greeks again. Finally,
it seems tempting to use these formulas for the computation of tail dependence
coefficients. However, due to numerical instabilities at the limits, they did
not yield any meaningful results.

\bibliographystyle{abbrvnat}
\bibliography{references}

\begin{thebibliography}{20}
\providecommand{\natexlab}[1]{#1}
\providecommand{\url}[1]{\texttt{#1}}
\expandafter\ifx\csname urlstyle\endcsname\relax
  \providecommand{\doi}[1]{doi: #1}\else
  \providecommand{\doi}{doi: \begingroup \urlstyle{rm}\Url}\fi

\bibitem[Barndorff-Nielsen(1998)]{Barndorff-Nielsen98}
O.~E. Barndorff-Nielsen.
\newblock {Processes of normal inverse Gaussian type}.
\newblock \emph{Finance Stoch.}, 2:\penalty0 41--68, 1998.

\bibitem[Bl{\ae}sild(1981)]{Blaesild81}
P.~Bl{\ae}sild.
\newblock {The two-dimensional hyperbolic distribution and related
  distributions, with an application to {J}ohannsen's bean data}.
\newblock \emph{Biometrika}, 68:\penalty0 251--263, 1981.

\bibitem[Cuchiero et~al.(2011)Cuchiero, Filipovi{\'c}, Mayerhofer, and
  Teichmann]{Cuchiero_Filipovic_Mayerhofer_Teichmann_2011}
C.~Cuchiero, D.~Filipovi{\'c}, E.~Mayerhofer, and J.~Teichmann.
\newblock {Affine processes on positive semidefinite matrices}.
\newblock \emph{Ann. Appl. Probab.}, 21:\penalty0 397--463, 2011.

\bibitem[Duffie et~al.(2003)Duffie, Filipovi{\'c}, and
  Schachermayer]{DuffieFilipovicSchachermayer03}
D.~Duffie, D.~Filipovi{\'c}, and W.~Schachermayer.
\newblock {Affine processes and applications in finance}.
\newblock \emph{Ann. Appl. Probab.}, 13:\penalty0 984--1053, 2003.

\bibitem[Eberlein and Madan(2010)]{EberleinMadan09}
E.~Eberlein and D.~Madan.
\newblock {On correlating {L}{\'e}vy processes}.
\newblock \emph{J. Risk}, 13:\penalty0 3--16, 2010.

\bibitem[Eberlein et~al.(2010)Eberlein, Glau, and
  Papapantoleon]{EberleinGlauPapapantoleon08}
E.~Eberlein, K.~Glau, and A.~Papapantoleon.
\newblock {Analysis of {F}ourier transform valuation formulas and
  applications}.
\newblock \emph{Appl. Math. Finance}, 17:\penalty0 211--240, 2010.

\bibitem[Hurd and Zhou(2010)]{HurdZhou09}
T.~R. Hurd and Z.~Zhou.
\newblock {A {F}ourier transform method for spread option pricing}.
\newblock \emph{SIAM J. Financial Math.}, 1:\penalty0 142--157, 2010.

\bibitem[Jacod and Protter(2003)]{Jacod_Protter_2003}
J.~Jacod and P.~Protter.
\newblock \emph{Probability Essentials}.
\newblock Springer, 2nd edition, 2003.

\bibitem[Kallsen and Tankov(2006)]{KallsenTankov04}
J.~Kallsen and P.~Tankov.
\newblock {Characterization of dependence of multidimensional L{\'e}vy
  processes using L{\'e}vy copulas}.
\newblock \emph{J. Multivariate Anal.}, 97:\penalty0 1551--1572, 2006.

\bibitem[Kawai(2009)]{Kawai09}
R.~Kawai.
\newblock {A multivariate {L}{\'e}vy process model with linear correlation}.
\newblock \emph{Quant. Finance}, 9:\penalty0 597--606, 2009.

\bibitem[Khanna and Madan(2009)]{Khanna_Madan:2009}
A.~Khanna and D.~Madan.
\newblock {Non {G}aussian models of dependence in returns}.
\newblock Preprint, \texttt{SSRN/1540875}, 2009.

\bibitem[Luciano and Schoutens(2006)]{LucianoSchoutens06}
E.~Luciano and W.~Schoutens.
\newblock {A multivariate jump-driven financial asset model}.
\newblock \emph{Quant. Finance}, 6:\penalty0 385--402, 2006.

\bibitem[Luciano and Semeraro(2010)]{Luciano_Semeraro_2010}
E.~Luciano and P.~Semeraro.
\newblock {A generalized normal mean-variance mixture for return processes in
  finance}.
\newblock \emph{Int. J. Theor. Appl. Finance}, 13:\penalty0 415--440, 2010.

\bibitem[McNeil et~al.(2005)McNeil, Frey, and Embrechts]{McNeilFreyEmbrechts05}
A.~McNeil, R.~Frey, and P.~Embrechts.
\newblock \emph{{Quantitative Risk Management: Concepts, Techniques and
  Tools}}.
\newblock Princeton University Press, 2005.

\bibitem[Muhle-Karbe et~al.(2012)Muhle-Karbe, Pfaffel, and
  Stelzer]{MuhleKarbe_Pfaffel_Stelzer_2012}
J.~Muhle-Karbe, O.~Pfaffel, and R.~Stelzer.
\newblock Option pricing in multivariate stochastic volatility models of {OU}
  type.
\newblock \emph{SIAM J. Financial Math.}, 3:\penalty0 66--94, 2012.

\bibitem[R{\"u}schendorf(2009)]{Rueschendorf_2009}
L.~R{\"u}schendorf.
\newblock {On the distributional transform, {S}klar's theorem, and the
  empirical copula process}.
\newblock \emph{J. Statist. Plann. Inference}, 139:\penalty0 3921--3927, 2009.

\bibitem[Sato(1999)]{Sato99}
K.~Sato.
\newblock \emph{{L{\'e}vy Processes and Infinitely Divisible Distributions}}.
\newblock Cambridge University Press, 1999.

\bibitem[Schmidt(2003)]{Schmidt03}
R.~Schmidt.
\newblock \emph{{Dependencies of Extreme Events in Finance}}.
\newblock PhD thesis, Univ. Ulm, 2003.

\bibitem[v.~Hammerstein(2011)]{Hammerstein_2011}
E.~A. v.~Hammerstein.
\newblock \emph{{Generalized Hyperbolic Distributions: {T}heory and
  Applications to {CDO} Pricing}}.
\newblock PhD thesis, Univ. Freiburg, 2011.

\bibitem[Villiger(2007)]{Villiger_2007}
S.~Villiger.
\newblock {Basket option pricing on sparse grids using fast {F}ourier
  transforms}.
\newblock Master's thesis, ETH Z{\"u}rich, 2007.

\end{thebibliography}

\end{document}